\newif\ifarXiv
\arXivfalse
\newif\ifWP
\WPfalse
\newif\ifLATIN
\LATINfalse

\arXivtrue		

\ifarXiv\LATINtrue\fi	

\newif\ifnotLATIN	
\notLATINtrue
\ifLATIN\notLATINfalse\fi

\ifarXiv
  \newcommand{\DFVIII}{vovk:arXiv0606093}
\fi
\ifWP
  \newcommand{\DFVIII}{GTP17}
\fi

\ifnotLATIN
  
\fi
\ifLATIN
  
\fi

\ifarXiv
\documentclass{article}
\usepackage{amsmath,amsfonts,amssymb,latexsym}
\newcommand{\Extra}[1]{}
\fi

\ifWP
\documentclass[toc]{gtarticle}
\usepackage{amsmath,amsfonts,amssymb,latexsym,epsfig}
\renewcommand{\Extra}[1]{#1}
\fi

\allowdisplaybreaks[4]

\ifnotLATIN
\input{OT2enc.def}

\usepackage{CJK}
\fi

\newcommand{\Vladimir}{Vladimir}
\newcommand{\DOT}{.}
\newcommand{\zzrelax}[1]{}

\newcommand{\K}{\mathcal{K}}

\newcommand{\bbbr}{\mathbb{R}}		

\newtheorem{theorem}{Theorem}
\newtheorem{corollary}{Corollary}
\newenvironment{proof}
  {\trivlist\item[\hskip\labelsep\textbf{Proof}]}
  {\endtrivlist}

\newcommand{\boxforqed}{\rule{.3em}{1.5ex}}
\newcommand{\qedtext}{\unskip\nobreak\hfil
  \penalty50\hskip1em\null\nobreak\hfil\boxforqed
  \parfillskip=0pt\finalhyphendemerits=0\endgraf}

\newenvironment{remark*}
  {\trivlist\item[\hskip\labelsep{\bfseries Remark}]\relax}
  {\endtrivlist}

\newlength{\IndentI}
\newlength{\IndentII}
\newlength{\IndentIII}
\newlength{\WidthI}
\newlength{\WidthII}
\newlength{\WidthIII}
\title{Hoeffding's inequality\\in game-theoretic probability}

\ifarXiv
\author{Vladimir Vovk\\
\texttt{vovk{\rm@}cs.rhul.ac.uk}\\
\texttt{http://vovk.net}}
\fi

\ifWP
\author{Vladimir Vovk}

\fi

\begin{document}
\maketitle
\begin{abstract}
  This note makes the obvious observation
  that Hoeffding's original proof of his inequality
  remains valid in the game-theoretic framework.
  All details are spelled out for the convenience of future reference.
\end{abstract}

\section{Introduction}

The game-theoretic approach to probability was started by von Mises
and greatly advanced by Ville \cite{ville:1939};
however, it has been overshadowed by Kolmogorov's measure-theoretic approach
\cite{kolmogorov:1933}.
The relatively recent book \cite{shafer/vovk:2001}
contains game-theoretic versions of several results of probability theory,
and it argues that the game-theoretic versions have important advantages
over the conventional measure-theoretic versions.
However, \cite{shafer/vovk:2001} does not contain
any large-deviation inequalities.
This note fills the gap by stating the game-theoretic version
of Hoeffding's inequality (\cite{hoeffding:1963}, Theorem 2).

\section{Hoeffding's supermartingale}

This section presents perhaps the most useful product of Hoeffding's method,
a non-negative supermartingale starting from 1.
This supermartingale will easily yield Hoeffding's inequality in the following section.

This is a version of the basic forecasting protocol from \cite{shafer/vovk:2001}:

\bigskip

\noindent
\textsc{Game of forecasting bounded variables}

\smallskip

\noindent
\textbf{Players:} Sceptic, Forecaster, Reality

\smallskip

\noindent
\textbf{Protocol:}

\parshape=6
\IndentI  \WidthI
\IndentI  \WidthI
\IndentII \WidthII
\IndentII \WidthII
\IndentII \WidthII
\IndentII \WidthII
\noindent
Sceptic announces $\K_0\in\bbbr$.\\
FOR $n=1,2,\dots$:\\
  Forecaster announces interval $[a_n,b_n]\subseteq\bbbr$ and number $\mu_n\in(a,b)$.\\
  Sceptic announces $M_n\in\bbbr$.\\
  Reality announces $x_n\in[a_n,b_n]$.\\
  Sceptic announces $\K_n \le \K_{n-1} + M_n (x_n - \mu_n)$.

\bigskip

\noindent
On each round $n$ of the game
Forecaster outputs an interval $[a_n,b_n]$ which, in his opinion,
will cover the actual observation $x_n$ to be chosen by Reality,
and also outputs his expectation $\mu_n$ for $x_n$.
The forecasts are being tested by Sceptic,
who is allowed to gamble against them.
The expectation $\mu_n$ is interpreted as the price of a ticket
which pays $x_n$ after Reality's move becomes known;
Sceptic is allowed to buy any number $M_n$,
positive, zero, or negative,
of such tickets.
When $x_n$ falls outside $[a_n,b_n]$,
Sceptic becomes infinitely rich;
without loss of generality
we include the requirement $x_n\in[a_n,b_n]$ in the protocol;
furthermore, we will always assume that $\mu_n\in(a_n,b_n)$.
Sceptic is allowed to choose his initial capital $\K_0$
and is allowed to throw away part of his money at the end of each round.

It is important that the game of forecasting bounded variables
is a perfect-information game:
each player can see the other players' moves
before making his or her (Forecaster and Sceptic are male and Reality is female) own move;
there is no randomness in the protocol.

A \emph{process} is a real-valued function
defined on all finite sequences
$(a_1,b_1,\mu_1,x_1,\ldots,a_N,b_N,\mu_N,x_N)$,
$N=0,1,\ldots$,
of Forecaster's and Reality's moves in the game of forecasting bounded variables.
If we fix a strategy for Sceptic,
Sceptic's capital $\K_N$, $N=0,1,\ldots$,
become a function of Forecaster's and Reality's previous moves;
in other words,
Sceptic's capital becomes a process.
The processes that can be obtained this way
are called (game-theoretic) \emph{supermartingales}.

The following theorem is essentially inequality (4.16) in \cite{hoeffding:1963}.
\begin{theorem}\label{thm:super}
  For any $h\in\bbbr$,
  the process
  \begin{equation*}
    \prod_{n=1}^N
    \exp
    \left(
      h(x_n-\mu_n)
      -
      \frac{h^2}{8} (b_n-a_n)^2
    \right)
  \end{equation*}
  is a supermartingale.
\end{theorem}
\begin{proof}
  Assume, without loss of generality,
  that Forecaster is additionally required
  to always set $\mu_n:=0$.
  (Adding the same constant to $a_n$, $b_n$, and $\mu_n$
  will not change anything for Sceptic.)
  Now we have $a_n<0<b_n$.

  It suffices to prove that on round $n$ Sceptic can make a capital of $\K$
  into a capital of at least
  \begin{equation*}
    \K
    \exp
    \left(
      h x_n
      -
      \frac{h^2}{8} (b_n-a_n)^2
    \right);
  \end{equation*}
  in other words,
  that he can obtain a payoff of at least
  \begin{equation*}
    \exp
    \left(
      h x_n
      -
      \frac{h^2}{8} (b_n-a_n)^2
    \right)
    -
    1
  \end{equation*}
  using the available tickets
  (paying $x_n$ and costing $0$).
  This will follow from the inequality
  \begin{equation*}
    \exp
    \left(
      h x_n
      -
      \frac{h^2}{8} (b_n-a_n)^2
    \right)
    -
    1
    \le
    x_n
    \frac{e^{h b_n}-e^{h a_n}}{b_n-a_n}
    \exp
    \left(
      -
      \frac{h^2}{8} (b_n-a_n)^2
    \right),
  \end{equation*}
  which can be rewritten as
  \begin{equation}\label{eq:goal}
    \exp
    \left(
      h x_n
    \right)
    \le
    \exp
    \left(
      \frac{h^2}{8} (b_n-a_n)^2
    \right)
    +
    x_n
    \frac{e^{h b_n}-e^{h a_n}}{b_n-a_n}.
  \end{equation}

  Our goal is to prove (\ref{eq:goal}).
  By the convexity of the function $\exp$,
  it suffices to prove
  \begin{equation*}
    \frac{x_n-a_n}{b_n-a_n}
    e^{h b_n}
    +
    \frac{b_n-x_n}{b_n-a_n}
    e^{h a_n}
    \le
    \exp
    \left(
      \frac{h^2}{8} (b_n-a_n)^2
    \right)
    +
    x_n
    \frac{e^{h b_n}-e^{h a_n}}{b_n-a_n},
  \end{equation*}
  i.e.,
  \begin{equation*}
    \frac
    {
      b_n e^{h a_n}
      -
      a_n e^{h b_n}
    }
    {b_n-a_n}
    \le
    \exp
    \left(
      \frac{h^2}{8} (b_n-a_n)^2
    \right),
  \end{equation*}
  i.e.,
  \begin{equation}\label{eq:simpler}
    \ln
    \left(
      b_n e^{h a_n}
      -
      a_n e^{h b_n}
    \right)
    \le
    \frac{h^2}{8} (b_n-a_n)^2
    +
    \ln(b_n-a_n).
  \end{equation}
  The derivative of the left-hand side of (\ref{eq:simpler}) is
  \begin{equation*}
    \frac
    {
      a_n b_n e^{h a_n}
      -
      a_n b_n e^{h b_n}
    }
    {
      b_n e^{h a_n}
      -
      a_n e^{h b_n}
    }
  \end{equation*}
  and the second derivative, after cancellations and regrouping, is
  \begin{equation*}
    (b_n-a_n)^2
    \frac
    {
      \left(
        b_n e^{h a_n}
      \right)
      \left(
        -a_n e^{h b_n}
      \right)
    }
    {
      \left(
        b_n e^{h a_n}
        -
        a_n e^{h b_n}
      \right)^2
    }.
  \end{equation*}
  The last ratio is of the form $u(1-u)$ where $0<u<1$.
  Hence it does not exceed $1/4$,
  and the second derivative itself does not exceed $(b_n-a_n)^2/4$.
  Inequality (\ref{eq:simpler}) now follows from the second-order Taylor expansion
  of the left-hand side around $h=0$.
  \qedtext
\end{proof}

The relation between the game-theoretic and measure-theoretic approaches to probability
is described in \cite{shafer/vovk:2001}, Chapter 8.
Intuitively,
the generality of the game-theoretic protocol
stems from the fact that Forecaster is not asked to produce
a full-blown probability forecast for $x_n$:
only the elements ($a_n,b_n,\mu_n$)
that we really need for our mathematical result
enter the game of forecasting bounded variables.
Besides, the players are allowed to react to each other moves;
in particular, Reality may react to Forecaster's moves
and both Reality and Forecaster may react to Sceptic's moves
(the latter is important in applications to defensive forecasting:
see, e.g., \cite{\DFVIII}).
It is remarkable that many measure-theoretic proofs
carry over in a straightforward manner to game-theoretic probability.

\section{Hoeffding's inequality}

We start from the definition of upper probability,
a game-theoretic counterpart (along with lower probability)
of the standard measure-theoretic notion of probability.
Suppose the game of forecasting bounded variables
lasts a known number $N$ of rounds.
(See \cite{shafer/vovk:2001} for the general definition.)
The \emph{sample space} is the set of all sequences
$(a_1,b_1,\mu_1,x_1,\ldots,a_N,b_N,\mu_N,x_N)$
of Forecaster's and Reality's moves in the game.
An \emph{event} is a subset of the sample space.
The \emph{upper probability} of an event $E$
is the infimum of the initial value of non-negative supermartingales
that take value at least $1$ on $E$.
(See \cite{shafer/vovk:2001}, Chapter 8,
for a demonstration that this definition agrees with measure-theoretic probability.)

Theorem \ref{thm:super} immediately gives Hoeffding's inequality
(cf.\ \cite{hoeffding:1963}, the proof of Theorem 2)
when combined with the definition of game-theoretic probability:
\begin{corollary}\label{cor:Hoeffding}
  Suppose the game of forecasting bounded variables
  lasts a fixed number $N$ of rounds.
  If all $a_n$ and $b_n$ are given in advance
  and $t>0$ is a known constant,
  the upper probability of the event
  \begin{equation}\label{eq:event}
    \frac1N
    \sum_{n=1}^N
    (x_n-\mu_n)
    \ge
    t
  \end{equation}
  does not exceed
  \begin{equation*}
    e^{-2N^2t^2/C},
  \end{equation*}
  where $C := \sum_{n=1}^N (b_n-a_n)^2$.
\end{corollary}
(The reader will see that it is sufficient for Sceptic to know only $C$
at the start of the game,
not the individual $a_n$ and $b_n$.)
\begin{proof}
  The supermartingale of Theorem \ref{thm:super} starts from $1$
  and achieves
  \begin{equation}\label{eq:achieves}
    \prod_{n=1}^N
    \exp
    \left(
      h(x_n-\mu_n)
      -
      \frac{h^2}{8} (b_n-a_n)^2
    \right)
    \ge
    \exp
    \left(
      hNt - \frac{h^2}{8}C
    \right)
  \end{equation}
  on the event~(\ref{eq:event}).
  The right-hand side of (\ref{eq:achieves}) attains its maximum at $h:=4Nt/C$,
  which gives the statement of the corollary.
  \qedtext
\end{proof}

\begin{remark*}
  The measure-theoretic counterpart of Corollary \ref{cor:Hoeffding}
  is sometimes referred to as the Hoeffding--Azuma inequality,
  in honour of Kazuoki Azuma
  \ifnotLATIN(\begin{CJK*}[dnp]{JIS}{min}¸ãºÊ °ì¶½\end{CJK*})\fi
  \cite{azuma:1967}.
  The martingale version, however, is also stated in Hoeffding's paper
  (\cite{hoeffding:1963}, the end of Section 2).
\end{remark*}

\subsection*{Acknowledgments}

This work is inspired by a question asked by Yoav Freund.
It has been partially supported by EPSRC (grant EP/F002998/1).

\end{document}